\documentclass[12pt]{amsart}
\usepackage{amstext,amsfonts,amssymb,amscd,amsbsy,amsmath,verbatim,fullpage}
\usepackage[alphabetic,lite,backrefs]{amsrefs} 
\usepackage[bookmarks, colorlinks=true, linkcolor=blue, citecolor=blue, urlcolor=blue, pagebackref=true]{hyperref}
\usepackage{ifthen}
\usepackage{color,tikz}
\usepackage{amsthm}
\usepackage{latexsym}
\usepackage[all]{xy}
\usepackage{enumerate}

\newtheorem{lemma}{Lemma}[section]

\newtheorem{cor}[lemma]{Corollary}

\newtheorem{claim*}{Claim}
\newtheorem{thm}[lemma]{Theorem}
\newtheorem*{thm*}{Theorem}

\newtheorem{defn}[lemma]{Definition}

\theoremstyle{remark}
\newtheorem{remark}[lemma]{Remark}
\newtheorem{example}[lemma]{Example}

\newcommand{\kk}{\mathbf k}

\newcommand{\FF}{\mathbb{F}}

\newcommand{\QQ}{\mathbb Q}

\newcommand{\ZZ}{\mathbb Z}

\author{John D. Berman}
\author{Daniel Erman}
\thanks{The first author was supported by an NSF Postdoctoral Fellowship under grant 1803089. The second author received support from NSF grant DMS-1601619.}
\title{Interpolation over $\ZZ$ and torsion in class groups}
\date{\today}
\begin{document} 

\begin{abstract}
We prove an interpolation result for homogeneous polynomials over the integers, or more generally for PIDs with finite residue fields. Previous proofs of this result use the well-known but nontrivial fact that class groups of rings of integers are torsion. In this short note, we provide an independent proof using elementary techniques.
\end{abstract}
\maketitle

Over a field, one can use homogeneous polynomials of bounded degree to solve any interpolation problem.  
Namely, given any points $P_1, \dots, P_s$ in $\kk^n$ and any specified values $a_1, \dots, a_s\in \kk$, one can always find a homogeneous polynomial $f\in \kk[x_1, \dots,x_n]$ of degree at most $s-1$ where $f(P_i)=a_i$ for all $i$; see, for instance, in~\cite[Theorem~4.2]{eis-syzygy}.

Interpolation questions become more complicated when one passes from a field to the integers or other commutative rings.  Given $P_1, \dots, P_s\in R^n$ and $a_1, \dots, a_s\in R$, one can ask whether there exists a homogeneous polynomial $f\in R[x_1, \dots, x_n]$ where $f(P_i)=a_i$ for all $i$.  Even over the integers, these problems can be quite subtle: not every such interpolation problem is solvable, and  even when a solution exists, one might need a polynomial of surprisingly high degree (see Examples~\ref{ex:deg 60} and \ref{ex:no solution}).

Our main result on interpolation over $\ZZ$ is that if $a_i=1$ for all $i$, and the coordinates of the $P_i$ are relatively prime, then one can always solve the interpolation problem.

\begin{thm}\label{thm:interpolation}
Suppose $S\subseteq \ZZ^n$ is a finite set where, for each $Q\in S$,  the coordinates $\langle q_1,\dots, q_n\rangle $ of $Q$ satisfy $\text{gcd}(q_1,\dots,q_n)=1$.  Then there exists a (nonconstant) homogeneous polynomial $f\in\ZZ[x_1,\dots,x_n]$ such that $f(Q)=1$ for each $Q\in S$.
\end{thm}

This theorem is not original. It arose in the context of recent results on Noether normalization, such as \cite{bruce-erman,cmbpt,gabber-liu-lorenzini}. However, all of these papers deduce Theorem~\ref{thm:interpolation} from the fact that class groups of rings of integers are torsion, a famous result from algebraic number theory, which is typically proven using analytic methods from the geometry of numbers. We could not help but ask whether this is overkill for a theorem with such a simple statement.  The main goal of this note is to provide a direct  proof of this interpolation result, which avoids facts about class groups.

Our methods lead us to also consider interpolation questions over other rings $R$.  In general, for a commutative ring $R$, we say:

\begin{defn}\label{defn:witness}
The $n$-tuple $\langle x_1,\dots,x_n\rangle\in R^n$ has \emph{coprime entries} if $(x_1,\dots,x_n)=(1)$ as ideals of $R$.
Suppose that for any $n\geq 1$ and finite set $S\subseteq R^n$ of points with coprime entries, there is a (nonconstant) homogeneous polynomial $f\in R[x_1,\dots,x_n]$ such that $f(X)$ is a unit for each $X\in S$. Then we say $R$ \emph{has enough homogeneous polynomials}.
\end{defn}

In fact, our proof of Theorem \ref{thm:interpolation} will apply verbatim to the more general situation:

\begin{thm}\label{thm:interpolation general}
If $R$ is a PID such that all quotients by maximal ideals are finite, then $R$ has enough homogeneous polynomials.
\end{thm}

This motivates the general question: Which rings have enough homogeneous polynomials? Since the first version of this paper appeared, there has already been additional work on this subject. Fresnel and Matignon~\cite{FM} have proven that the following conditions are equivalent:
\begin{itemize}
\item $R$ has enough homogeneous polynomials.
\item $R$ has enough homogeneous polynomials in 2 variables.
\item For each $a\in R$, $\text{coker}(R^\times\rightarrow(R/a)^\times)$ is torsion. (\cite{KLW} say that $R$ has `power stable range 1'.)
\item $\mathcal{O}(1)$ is a torsion element of $\text{Pic}\,\text{Proj}\,R[x,y]/(x(ay-bx))$ whenever $\text{gcd}(a,b)=1$.
\end{itemize}

This result confirms that the condition `$R$ has enough homogeneous polynomials' is both testable and intimately related to torsion in Picard groups. Gabber, Liu, and Lorenzini~\cite{gabber-liu-lorenzini} have studied rings of pictorsion (that is, rings for which all finite algebras have torsion Picard groups) in their study of closed subschemes of projective schemes. Fresnel and Matignon~\cite{FM} prove that all pictorsion rings have enough homogeneous polynomials. Hence, our techniques may be applicable to studying torsion in $\text{Pic}(R)$ when $R$ is more general than a ring of integers.

We would like to highlight some subtleties in this subject, by means of a few examples.

\begin{example}\label{ex:deg 60}
By Theorem \ref{thm:interpolation}, there is a homogeneous polynomial $f\in\ZZ[x,y]$ such that \begin{equation*}\begin{array}{ccc}
f(1,4)=1 &f(3,5)=1 &f(4,5)=1.
\end{array}\end{equation*} In fact, there is such an $f$ of degree $d$ if and only if $d$ is divisible by $60$! One example is \[\scriptstyle
(11x^6-43x^5y+14x^4y^2+71x^3y^3-82x^2y^4+32xy^5-4y^6)^{10}-y(4x-y)(5x-3y)(x-y)^{50}(2x^7-x^6y-x^5y^2+2x^4y^3-2xy^6+y^7),
\] and any power of this polynomial works as well. It is not easy to see that no lower degree polynomial will work, but this can be confirmed via direct calculation in a computer algebra package such as Macaulay2~\cite{M2}.\footnote{The computation is the following: for any $d$, let $M_d$ be the $3\times(d+1)$ matrix where the 3 entries in column $j$ are obtained by evaluating the monomial $x^{d-j}y^j$ at $\langle 1,4\rangle, \langle 3,5\rangle$ and $\langle 4,5\rangle$ respectively.  There exists a homogeneous polynomial of degree $d$ with the desired property if and only if the vector $\langle 1,1,1 \rangle$ lies in the image of $M_d$, and one can check this by computing the Smith Normal Form of $M_d$.}
\end{example}

\begin{example}\label{ex:no solution}
It is not true for any $a,b,c\in\ZZ$ that there is homogeneous $f$ satisfying \begin{equation*}\begin{array}{ccc}
f(1,4)=a &f(3,5)=b &f(4,5)=c.
\end{array}\end{equation*} For example, if exactly one of $b,c$ is divisible by $5$, there can be no such polynomial.
\end{example}

\begin{example}
We have said that $R$ has enough homogeneous polynomials if, for any finite set $S\subseteq R^n$ with coprime entries, there is a homogeneous $f$ with $f(X)\in R^\times$ for all $X\in S$.

We can also ask the stronger question: Is there a homogeneous $f$ for which $f(X)=1$ for all $X\in S$? If $R^\times$ is torsion (as it is for $\ZZ)$, then the two questions are equivalent. 

On the other hand, if $R^\times$ is not torsion, then $R$ will never satisfy the stronger condition. For instance, if $u\in R^\times$ has infinite order, then there can be no (nonconstant) homogeneous polynomial $f$ which evaluates to $1$ on both $\langle 1,1\rangle $ and $\langle u,u\rangle $.
\end{example}

\begin{example}
An example of a PID which does not have enough homogeneous polynomials is $\QQ[t]$. (The set $S=\{\langle 1,1+t\rangle,\langle 1-t,1\rangle\}$ will fail; see~\cite[Example 8.5]{bruce-erman}.) Another example is $\FF_q[s,t]$ ~\cite[Example 8.6]{bruce-erman}.
\end{example}

\begin{example}\label{ex:converse}
Goldman~\cite{Goldman} constructed a Dedekind domain $A$ such that $\mathbb{Z}[x]\subseteq A\subseteq\mathbb{Q}[x]$, all quotients by maximal ideals are finite, yet the Picard group is not torsion. Fresnel and Matignon~\cite{FM} point out that $A$ has enough homogeneous polynomials but does not have torsion Picard group.
\end{example}

We end with a sketch of the proof idea and some of the colorful history of this problem.

To prove that $\ZZ$ has enough homogeneous polynomials, the idea is to reduce to studying homogeneous polynomials over $\ZZ/a\ZZ$. Then we leverage the fact that $(\ZZ/a\ZZ)^n$ is finite, so that there are enough homogeneous polynomials for essentially combinatorial reasons.

\begin{remark}
An early draft of this paper claimed that Theorem~\ref{thm:interpolation} could be used to prove the class groups of rings of integers were torsion.  While our argument was flawed, it would be interesting to better understand this connection.  More broadly, it would be interesting to see if one can one prove that class groups of rings of integers are torsion without relying on methods from the geometry of numbers.
\end{remark}

The special case $n=2$ of Theorem~\ref{thm:interpolation} was submitted by the first author to the International Math Olympiad and appeared as a problem in 2017.  It was completely solved by just 14 students, making it one of the most challenging problems to appear at the IMO.  The Olympiad solutions yielded two distinct proof strategies for Theorem \ref{thm:interpolation}: one of those aligns with our strategy.  A second strategy, first written down by Dan Carmon, is to induct on the size of the set $S$. See~\cite[pp. 85--87]{imo}.

\section*{Acknowledgments}
We thank Craig Huneke, who first raised the question of finding an elementary proof of Theorem~\ref{thm:interpolation}, and whose ideas and suggestions were very helpful.  We also thank Jordan Ellenberg, Trevor Hyde, Niranjan Ramachandran, Ravi Vakil, Bianca Viray, and Melanie Matchett Wood for useful conversations.  We thank Aurel Page for pointing out a subtle error in an earlier draft.

\section{Main Results}\label{sec:interpolation}
We first observe that fields have enough homogeneous polynomials; this is well-known (see for example~\cite[Theorem~4.2]{eis-syzygy}), but we provide a proof for completeness.

\begin{lemma}\label{lem:field}
Any field has enough homogeneous polynomials.
\end{lemma}

\begin{proof}
For each point $v\in S$, let $I_v$ denote the ideal generated by homogeneous polynomials vanishing at $v$. Equivalently, $I_v$ is the ideal of polynomials vanishing at a line through $v$, so it is prime. The ideal $\mathfrak{m}=(x_0,\ldots,x_n)$ is not contained in any single $I_v$, so by prime avoidance~\cite[Lemma~3.3]{eisenbud} there is some homogeneous polynomial $P$ in $\mathfrak{m}$ which is not in $\bigcup_v I_v$; that is, $P(v)\neq 0$ for all $v\in S$.
\end{proof}

\begin{lemma}\label{lem:enough} We have the following:
\begin{enumerate}
	\item\label{lem:enough:1}  Suppose $I\subseteq R$ is a nilpotent ideal, and $R/I$ has enough homogeneous polynomials. Then so does $R$.
	\item\label{lem:enough:2}  If $R_1$ and $R_2$ have enough homogeneous polynomials, then so does $R_1\times R_2$.
\end{enumerate}
\end{lemma}

\begin{proof}
For \eqref{lem:enough:1}: Take $S$ as in Definition \ref{defn:witness}. Because $R/I$ has enough homogeneous polynomials, there is some homogeneous $f\in R[x_1,\dots,x_n]$ such that $f(v)$ is a unit mod $I$ for each $v$. However, $x\in R$ is a unit if and only if its residue mod $I$ is a unit (because $I$ is nilpotent). Therefore, $f(v)$ is a unit in $R$ for each $v$, so $R$ has enough homogeneous polynomials.

For \eqref{lem:enough:2}: Let $R=R_1\times R_2$. For $S\subseteq R^n$ as in Definition \ref{defn:witness}, let $S_1,S_2$ be the projections onto $R_1^n,R_2^n$. Since $R_1,R_2$ have enough homogeneous polynomials, there are homogeneous $f_1\in R_1[x_1,\dots,x_n]$ and $f_2\in R_2[x_1,\dots,x_n]$ such that $f_1(v)$ is a unit for all $v\in S_1$ and $f_2(v)$ is a unit for all $v\in S_2$. 

Choose positive integer exponents $e_1,e_2$ such that $f_1^{e_1}$ and $f_2^{e_2}$ have the same degree $d$. 
The splitting $R=R_1\times R_2$ induces a corresponding splitting of $R[x_1,\dots,x_n]$, and thus the pair $(f_1^{e_1},f_2^{e_2})$ naturally defines a homogeneous  element of $R[x_1,\dots,x_n]$ where
$f(v)=(f_1(v),f_2(v))$ for all $v\in S$.  Therefore $R$ has enough homogeneous polynomials.
\end{proof}

\begin{cor}\label{lem:artinian}
Any Artinian ring has enough homogeneous polynomials.  In particular, $\ZZ/a\ZZ$ has enough homogeneous polynomials for any $a\geq 1$.
\end{cor}

\begin{proof}
Any Artinian ring is a finite product of local Artinian rings~\cite[Corollary~2.16]{eisenbud}, and thus by Lemma~\ref{lem:enough}\eqref{lem:enough:2} we can reduce to considering a local Artinian ring.  However, in any local Artinian ring the maximal ideal is nilpotent, and thus it follows from Lemma~\ref{lem:enough}\eqref{lem:enough:1} and Lemma~\ref{lem:field}.
\end{proof}

In order to prove Theorem \ref{thm:interpolation}, we only need to know $\ZZ/a\ZZ$ has enough homogeneous polynomials for each $n$. This is easier than the above proofs may suggest. For example,

\begin{example}
Consider $S=\{\langle 2,3\rangle,\langle5,7 \rangle,\langle 11,13\rangle \}$ in $\ZZ/100\ZZ$. We know that $\ZZ/100\ZZ$ is a product of nilpotent extensions of $\ZZ/5\ZZ$ and $\ZZ/2\ZZ$. By Lemma \ref{lem:enough}, it suffices to find homogeneous polynomials over $\ZZ/5\ZZ$ and $\ZZ/2\ZZ$ which evaluate to units on $S$.

We may pick $x^4+y^4$ over $\ZZ/5\ZZ$ and $x^2+xy+y^2$ over $\ZZ/2\ZZ$. As in Lemma~\ref{lem:enough}\eqref{lem:enough:2}, the polynomial $(x^4+y^4,(x^2+xy+y^2)^2)=(1,1)x^4+(0,2)x^3y+(0,3)x^2y^2+(0,2)xy^3+(1,1)y^4$ over $\ZZ/25\ZZ\oplus\ZZ/4\ZZ$ works. The Chinese Remainder Theorem exhibits $\ZZ/25\ZZ\oplus\ZZ/4\ZZ\cong\ZZ/100\ZZ$ via the identification $(a,b)\mapsto 76a+25b$, so that the polynomial $x^4+50x^3y+75x^2y^2+50xy^3+y^4$ evaluates to units on $S$. (In fact, it evaluates to a unit on \emph{any} vector with coprime entries, so the choice of $S$ is irrelevant.)
\end{example}

\begin{proof}[Proof of Theorem~\ref{thm:interpolation}]
We prove that $\ZZ$ has enough homogeneous polynomials.
Let $S\subseteq\ZZ^n$ be a finite set of vectors, each one with coprime entries. If $f$ is a homogeneous polynomial, then $f(v)=\pm f(-v)$. Therefore, we can assume without loss of generality that $S$ does not contain any two vectors that are scalar multiples of each other.

For each $v\in S$, there is some $f_v\in\ZZ[x_1\dots,x_n]$, homogeneous of degree $d_v$, such that $f_v(v)\neq 0$ and $f_v(w)=0$ for $w\neq v$ in $S$. Choose once and for all such a $f_v$ for each $v\in S$.

Let $a=\prod_v f_v(v)$. Since $\ZZ/a\ZZ$ has enough homogeneous polynomials (Corollary \ref{lem:artinian}), there is some homogeneous $h\in\ZZ[x_1,\dots,x_n]$ such that
\[
h(v)\equiv 1\text{ }(\text{mod }a)
\] for each $v\in S$. (Take a polynomial such that $h(v)$ is a unit mod $a$ for each $v$, then raise it to the power $\phi(a)$.) Choose an exponent $h^k$ such that $\text{deg}(h^k)=d\geq d_v$ for all $v\in S$.

Since each $v$ has coprime entries, we may pick linear homogeneous polynomials $L_v$ for which $L_v(v)=1$ (one for each $v\in S$). Let \[
g_v=\frac{a}{f_v(v)}L_v^{d-d_v}f_v,
\] which is a homogeneous polynomial of degree $d$ with integer coefficients. Moreover, $g_v(v)=a$, while $g_v(w)=0$ for any $w\neq v$.

To complete the proof, the linear combination \[
h-\sum_{v\in S}\frac{h(v)-1}{a}g_v
\] is homogeneous of degree $d$ and evaluates to $h(v)-\frac{h(v)-1}{a}a=1$ at each $v\in V$. Therefore, $\ZZ$ has enough homogeneous polynomials.
\end{proof}

\begin{cor}\label{cor:enough}
Suppose that $R$ is an integral domain for which every quotient by a nonzero principal ideal has enough homogeneous polynomials, and every group $\text{coker}(R^\times\rightarrow(R/I)^\times)$ is torsion. Then $R$ has enough homogeneous polynomials.
\end{cor}

\begin{proof}
The proof exactly follows that of Theorem~\ref{thm:interpolation}. We must have $R$ a domain to ensure $a\neq 0$, and torsion cokernels to ensure that we can choose $h$ which satisfies: for each $v\in S$, there is a unit $u_v\in R^\times$ such that $h(v)\equiv u_v\text{ }(\text{mod }a)$.
\end{proof}

We obtain Theorem~\ref{thm:interpolation general} immediately:

\begin{proof}[Proof of Theorem~\ref{thm:interpolation general}]
We are to prove that $R$ has enough homogeneous polynomials. If $R$ is a PID where all quotients by maximal ideals are finite fields, then for any nonzero ideal $I\subseteq R$, $R/I$ will be a finite ring.   In particular $(R/I)^\times$ is a finite set, and thus $R$ has enough homogeneous polynomials by Corollary~\ref{cor:enough}.
\end{proof}

\begin{remark}
The condition that $\text{coker}(R^\times\rightarrow(R/I)^\times)$ is torsion for all principal ideals $I$ can be restated: Any $x,y\in R$ with $\text{gcd}(x,y)=1$ and $y\neq 0$ satisfy $x^n+ay\in R^\times$ for some $n\geq 1$ and $a\in R$. In this case, Khurana, Lam, and Wang \cite{KLW} say that $R$ has \emph{power stable range one}.

In their follow-up to this paper, Fresnel and Matignon \cite{FM} prove that $R$ has enough homogeneous polynomials if and only if it has power stable range one, extending Corollary \ref{cor:enough}. We remark that, if $R$ is also a finite dimensional integral domain, this result follows from Corollary \ref{cor:enough} by induction on dimension.
\end{remark}

\end{document}